\documentclass[a4paper, 12pt, roman, titlepage]{article}
\usepackage{MRHStyle}
\usepackage{JournalOfExpMath}

\usepackage[square]{natbib}
\setcitestyle{citesep={;}}

\renewcommand{\phi}{\ensuremath{\varphi}}

\begin{document}
\title{\Large \textbf{Bounds for zeros of Collatz polynomials, with necessary and sufficient strictness conditions}}
\author{\Large Matt Hohertz \\ Department of Mathematics, Rutgers University \\ matt.hohertz@rutgers.edu \\ ORCID 0000-0001-6724-1034}

\maketitle

\begin{abstract}
	In a previous paper, we introduced the Collatz polynomials $P_N\inps{z}$, whose coefficients are the terms of the Collatz sequence of the positive integer $N$.  Our work in this paper expands on our previous results, using the Eneström-Kakeya Theorem to tighten our old bounds of the roots of $P_N\inps{z}$ and giving precise conditions under which these new bounds are sharp.  In particular, we confirm an experimental result that zeros on the circle $\sbld{z\in\C:\aval{z} = 2}$ are rare: the set of $N$ such that $P_N\inps{z}$ has a root of modulus 2 is sparse in the natural numbers.  We close with some questions for further study.
\end{abstract}

\textsc{keywords} Polynomial zeros, Collatz conjecture, generating functions \par
\textsc{word count} 1295
\section*{Biographical note}
Matt Hohertz received his Ph.D. in 2022 from Rutgers University.  His thesis, titled \ic{Expanding the Geometric Modulus Principle}, explores the behavior of analytic functions near critical points.  His recent research is in one-variable complex analysis, harmonic function theory, and polynomial root isolation.

\section{Introduction}
Let $c(N)$ be the Collatz iterate of a number $N\in\N\cup\sbld{0}$: \ic{i.e.},
\begin{equation}
	c(N) := \begin{cases}
		\frac{N}{2}, &N\mbox{ even} \\
		\frac{3N+1}{2}, &N\mbox{ odd and not $1$} \\
		0, &N = 1.
	\end{cases}
\end{equation}
Also, let $c^j(N)$ have the standard meaning
\begin{equation}
	c^j(N) := \begin{cases}
		N, &j = 0 \\
		c\inps{c^{j-1}(N)}, &j \geq 1,
	\end{cases}
\end{equation}
and define $n(N)$ (just $n$ when $N$ is clear from context) as 
\begin{equation} n(N):=\min \sbld{j\in\N : c^j(N) = 1}.
\end{equation} 
In this paper, we assume\footnote{\ic{cf.} item \ref{item:polyn} of Section \ref{sec:final-remarks}} that $n\inps{N} < \infty$ for all $N$, even though this remains an open question as of January 2022 \citep{boas}.  As in \citep{hohertz_kalantari}, we define the \ic{$N^{th}$ Collatz polynomial} to be
\begin{equation}
	P_N(z) := \sum_{j=0}^{n\inps{N}} c^j(N)\cdot z^j,
\end{equation}
\ic{i.e.}, the polynomial whose coefficients are the Collatz iterates of $N$, or equivalently consecutive members of the \ic{Collatz trajectory/sequence of $N$} (we will assume $N\in\N - \sbld{0,1}$ to avoid the trivial $P_0\inps{z} = 0$ and $P_1(z) = 1$).  Throughout the article, let $z_N$ signify an arbitrary root of $P_N(z)$ (which we call a \ic{Collatz zero}), and let $[k] := [1,k] \cap \N$. \par
We prove that
\begin{equation}
	\frac{2\cdot M(N)}{3\cdot M(N)+1} \leq \aval{z_N} \leq 2,
\end{equation} where $M(N)$ is the least odd iterate of $N$ other than 1.  In fact, we go somewhat further, proving that
\begin{equation}
	\frac{2\cdot M(N)}{3\cdot M(N)+1} < \aval{z_N} < 2
\end{equation} for almost all $N$ and giving precise conditions under which $N$ belongs to the sparse set of exceptions.

\section{General bounds}
\begin{lemma}[Eneström-Kakeya, Theorem A of \citep{eksharp}] \label{lemma:ek}
	Let $f(z) = a_kz^k + \cdots + a_1z + a_0$ have all strictly positive coefficients and set
	\begin{align}
		\alpha[f] &:= \min_{j=0,\cdots,k-1} \frac{a_j}{a_{j+1}} \\
		\beta[f] &:= \max_{j=0,\cdots,k-1} \frac{a_j}{a_{j+1}}.
	\end{align}
	Then
	\begin{equation} \label{eq:ek}
		\alpha[f] \leq \aval{w} \leq \beta[f]
	\end{equation} for all roots $w$ of $f(z)$.
\end{lemma}
With Lemma \ref{lemma:ek}, we prove the following general bound for $\aval{z_N}$:
\begin{theorem} \label{thm:bounds}
	\begin{equation} \label{eq:bounds}
		\frac{2\cdot M\inps{N}}{3\cdot M\inps{N} + 1} \leq \aval{z_N} \leq 2,
	\end{equation}
	where $M\inps{N}:=\max\Big(\sbld{-1/2}\cup\sbld{\ell > 1\::\:\ell\mbox{ an odd Collatz iterate of }N}\Big)$.
	\begin{proof}
		If $a_j$ is odd and not 1, then
		\begin{equation}
			\frac{a_j}{a_{j+1}} = \frac{2a_j}{3a_j + 1},
		\end{equation} and if $a_j$ is even, then
		\begin{equation}
			\frac{a_j}{a_{j+1}} = \frac{2a_j}{a_j} = 2.
		\end{equation}
		If $N$ has an odd iterate other than 1, then the conclusion follows by Lemma \ref{lemma:ek}.  Otherwise, $P_N(z)$ is the partial sum
		\begin{equation}
			N\cdot \inps{1 + \frac{z}{2} + \cdots + \frac{z^n}{2^n}},
		\end{equation}
		whose roots all lie on $\sbld{z\in\C : \aval{z} = 2}$.
	\end{proof}
\end{theorem}

\section{Strictness of the lower bound}
\begin{lemma}[Theorem B of \citep{eksharp}] \label{lemma:thmb}
	For a polynomial $f(z)$ of degree $k$ with strictly positive coefficients, let $S[f]$ be the set of all $j\in[k+1]$ such that
	\begin{equation} \label{eq:strin}
		\beta\left[f\right] > \frac{a_{k-j}}{a_{k+1-j}}.
	\end{equation}Then the upper bound of Equation \eqref{eq:ek} is an equality if and only if $1 < d := \gcd\inps{j\in S[f]}$, in which case
	\begin{itemize}
		\item all the zeros on $\aval{z} = \beta[f]$ are simple and given by $\Big\{\beta[f]\cdot \exp\inps{\frac{2\pi i j}{d}},\;j=1,\cdots, d-1\Big\}$ and
		\item $f\inps{\beta[f]\cdot z} = \inps{1 + z + \cdots + z^{d-1}}\cdot q_m\inps{z^d}$ for a degree $m$ polynomial $q_m$ with strictly positive coefficients.
	\end{itemize} 
	Moreover, if $m > 0$, then all zeros of $q_m$ belong to $\D$ and $\beta[q_m] \leq 1$.
\end{lemma}
Define the \ic{reciprocal polynomial} $\tilde{f}$ of the degree $k$ polynomial $f(z) = a_kz^k + \cdots + a_jz^j + \cdots + a_0$, for which we assume $a_0a_k\neq 0$, to be 
\begin{align}
	\tilde{f}(z) &:= z^k\cdot f\inps{\frac{1}{z}} \\
	&= a_0z^k + \cdots + a_{k-j}z^j + \cdots + a_k.
\end{align}
Note that 
\begin{align}
	\beta[\tilde{f}] &= \max_{j=0,\cdots, k-1}\frac{a_{k-j}}{a_{k-j-1}} \\
	&=\frac{1}{ \min_{j=0,\cdots, k-1}\frac{a_{j}}{a_{j+1}}} \\
	&= \frac{1}{\alpha[f]},
\end{align} and that, more generally, an upper bound for the zeros of $\tilde{f}$ is the reciprocal of a lower bound for the zeros of $f$.
\begin{theorem} \label{thm:lb}
	Equality holds for the leftmost inequality of Equation \eqref{eq:bounds} if and only if $N$ is a power of 2.
	\begin{proof}
		We have observed the ``if" direction in the proof of Theorem \ref{thm:bounds}.  For the ``only if" direction, suppose that $N$ is not a power of 2 and let $M\inps{N}$ again signify the minimum odd iterate of $N$ not equal to 1.  Then
		\begin{equation}
			 \beta\left[\tilde{P_N}\right] =  \frac{1}{\alpha\left[P_N\right]} = \frac{3}{2} + \frac{1}{2\cdot M\inps{N}}.
		\end{equation}
	Now, for $j\in[n+1]$,
	\begin{equation}
		\frac{a_{n-\inps{n-j}}}{a_{n-\inps{n+1-j}}} = \frac{a_j}{a_{j-1}} = 
		\begin{cases} 
			\frac{1}{2}, &a_{j-1}\mbox{ even and }j\neq n+1 \\ 
			\frac{3}{2} + \frac{1}{2a_{j-1}}, &a_{j-1}\mbox{ odd} \\
			0, &j = n+1,
		\end{cases}
	\end{equation}
	so that $\beta\left[\tilde{P_N}\right] > \frac{a_j}{a_{j-1}}$ for all but the unique $j$ for which $a_{j-1} = M$.  Thus, there exist at least two consecutive values of $j$ such that $\beta\left[\tilde{P_N}\right] > \frac{a_j}{a_{j-1}}$, implying that $\gcd\inps{j\in S[\tilde{P_N}]} = 1$.  Therefore, by Lemma \ref{lemma:thmb}, $\beta\left[\tilde{P_N}\right]$ is a non-sharp upper bound for the zeros of $\tilde{P_N}$, whence $\alpha\left[P_n\right]$ is a non-sharp lower bound for the zeros of $P_N$.
	\end{proof}
\end{theorem}

\section{Strictness of the upper bound}
For a given $N$, define the set $T_N:= \sbld{n+1}\cup\sbld{j\in[n] : c^{n-j}(N)\mbox{ is odd}}$.
Then Lemma \ref{lemma:thmb} implies the following sharpness result:
\begin{theorem} \label{thm:s}
	Let $d_N := \gcd(j\in T_N)$.  Then the upper bound of Equation \eqref{eq:bounds} is an equality if and only if $d_N > 1$, in which case
	\begin{enumerate}
		\item the zeros of $P_N$ on $\sbld{\aval{z} = 2}$ are simple and given precisely by $\sbld{2\omega : \omega^{d_N} = 1\land \omega\neq 1}$ and
		\item $P_N$ factors as \begin{equation} P_N(2z) = \inps{1 + z + \cdots + z^{d_N-1}}\cdot Q_N(z^{d_N}),\end{equation}
		where $Q_N$ is a polynomial with positive coefficients.  Moreover, if $Q_N$ is non-constant then all its zeros lie in $\D$ and $\beta[Q_N] \leq 1$.
	\end{enumerate}
\end{theorem}
\begin{proof}
	By the proof of Theorem \ref{thm:bounds} (and the fact that $a_{-1} = 0$ because $P_N$ is a polynomial), the set $T_N$ is precisely the set of $j\in[n+1]$ such that $\frac{a_{n-j}}{a_{n+1-j}} < \beta[P_N] = 2$.  Thus the result follows from applying Lemma \ref{lemma:thmb}.
\end{proof}
The exacting conditions of Theorem \ref{thm:s} suggest that zeros on $\aval{z} = 2$ are rare, and indeed we prove that this is so.  First, we need two lemmas.
\begin{lemma} \label{lemma:cons-ones}
	The number of length $k$ binary strings with no consecutive 1s is $F_{k+1}$, the $\inps{k+1}^{th}$ Fibonacci number.  In particular, let $p_k$ be the probability that a length $k$ binary string selected uniformly at random contains no consecutive 1s.  Then $\lim_{k\rightarrow\infty} p_k = 0$.
\end{lemma}
\begin{proof}
	The first statement holds for $k = 0$ and $1$.  Consider a binary string $x$ of length $k\geq 2$ and let $substring(x, j)$ be the substring of $x$ comprising its first $j$ digits.  If the last digit of $x$ is 0, then $x$ has two consecutive 1s if and only if $substring(x, k-1)$ does.  Otherwise, its last two digits are either $11$, in which case it has consecutive 1s, or $01$, in which case it has consecutive 1s if and only if $substring(x, k-2)$ does.  Therefore, the number of length $k$ binary strings with no consecutive 1s satisfies the same recurrence relation as the Fibonacci numbers. \par
	Since $F_k\sim \frac{\phi^k}{\sqrt{5}}$, $p_k\sim \frac{\phi}{\sqrt{5}}\inps{\frac{\phi}{2}}^k$, and this quantity has limit 0 as $k\rightarrow\infty$.
\end{proof}
\begin{lemma} \label{lemma:sparse-subset}
	Let $X$ be the set of natural numbers $N$ with the property that
	$$c^j\inps{N}\mbox{ odd}\Rightarrow c^{j+1}\inps{N}\mbox{ even}$$
	for all integers $j\geq 0$.  The set $X$ has density zero in the natural numbers.
\end{lemma}
\begin{proof}
	Let $x_j\inps{N}:= c^j\inps{N}\;\inps{\mbox{mod }2}.$ By \citep[Theorem B]{lagarias_1985}, the function $\Z\rightarrow {\Z/2^{k+1}\Z}$ defined by
	\begin{equation}
		N\rightarrow \inps{x_0\inps{N},\dots, x_k\inps{N}}
	\end{equation}
	is periodic with period $2^{k+1}$.  For a fixed $k$, the set $X$ is a subset of the preimage of the set of length $k+1$ strings with no consecutive 1s under this function.  By Lemma \ref{lemma:cons-ones}, as $k\rightarrow\infty$, this preimage becomes sparse as a subset of the natural numbers.
\end{proof}
Lemmas \ref{lemma:cons-ones} and \ref{lemma:sparse-subset} culminate in the following theorem.
\begin{theorem}\label{thm:zerodens}
	The set of $N$ such that $P_N$ has a root on $\aval{z} = 2$ has density 0 in the natural numbers.
\end{theorem}
\begin{proof}
	By Lemma \ref{lemma:sparse-subset}, $T_N$ contains consecutive natural numbers for almost all $N$, so that $\gcd\inps{j\in T_N} = 1$ for almost all $N$.  The conclusion follows from Theorem \ref{thm:s}.
\end{proof}

\section{Final remarks} \label{sec:final-remarks}
We conclude with a few suggestions for further research.
\begin{enumerate}
	\item Strengthen Theorem \ref{thm:lb} by finding an explicit, closed-form lower bound for $\aval{z_N}$ in terms of $N$.
 	\item Strengthen Theorems  \ref{thm:s} and \ref{thm:zerodens} by finding an explicit, closed-form upper bound for $\aval{z_N}$ in terms of $N$.
	\item Figure \ref{fig:zeros} suggests that certain subsets of $\sbld{z\in\C\::\:\aval{z} \leq 2}$ contain large clusters of Collatz zeros while others are zero-free.  Give an algorithm for proving that a subset of $\sbld{z\in\C\::\:\aval{z} \leq 2}$ is free of Collatz zeros.
	\item While Descartes' Rule of Signs implies that no Collatz zero is positive real, there appears to be a sequence of zeros that approaches $z = 1$ arbitrarily closely, bounded by a parabola with vertex at $z = 1$.  Prove or disprove the existence of such a convergent sequence and parabola.
	\item Under some modest assumptions, the zeros of a polynomial with random coefficients cluster near the unit circle with uniform angular distribution \citep{hughes_nikeghbali_2008}.  Prove or disprove that these assumptions hold for the Collatz polynomials.
	\item Find the Galois groups of $P_N$ for general classes of $P_N$.
	\item Expanding on Theorem \ref{thm:s}, find conditions under which $P_N$ has zeros at real multiples (other than 2) of roots of unity.
	\item Except for Theorem \ref{thm:lb}, all the theorems in this paper are functions of the Collatz sequence of $N$ rather than of $N$ itself.  Prove theorems on $z_N$ which can be applied without calculating the Collatz sequence of $N$ (\ic{e.g.}, if $N\equiv 3\mbox{ mod }4$, then the zeros of $P_N$ lie in ...).
	\item Our results rest on the assumption, equivalent to the Collatz conjecture itself, that $P_N$ is a polynomial for all $N$.  Using other methods of proof, find a property of Collatz zeros that contradicts a theorem in this paper, thereby disproving the Collatz conjecture. \label{item:polyn}
\end{enumerate}
\section*{Acknowledgements}
I would like to thank Harold Boas for his suggestions.
\section*{Declaration of interests}
The authors report there are no competing interests to declare.
\section*{Figures}
	\begin{figure}[h]
	\centering
	\includegraphics[width=1\textwidth]{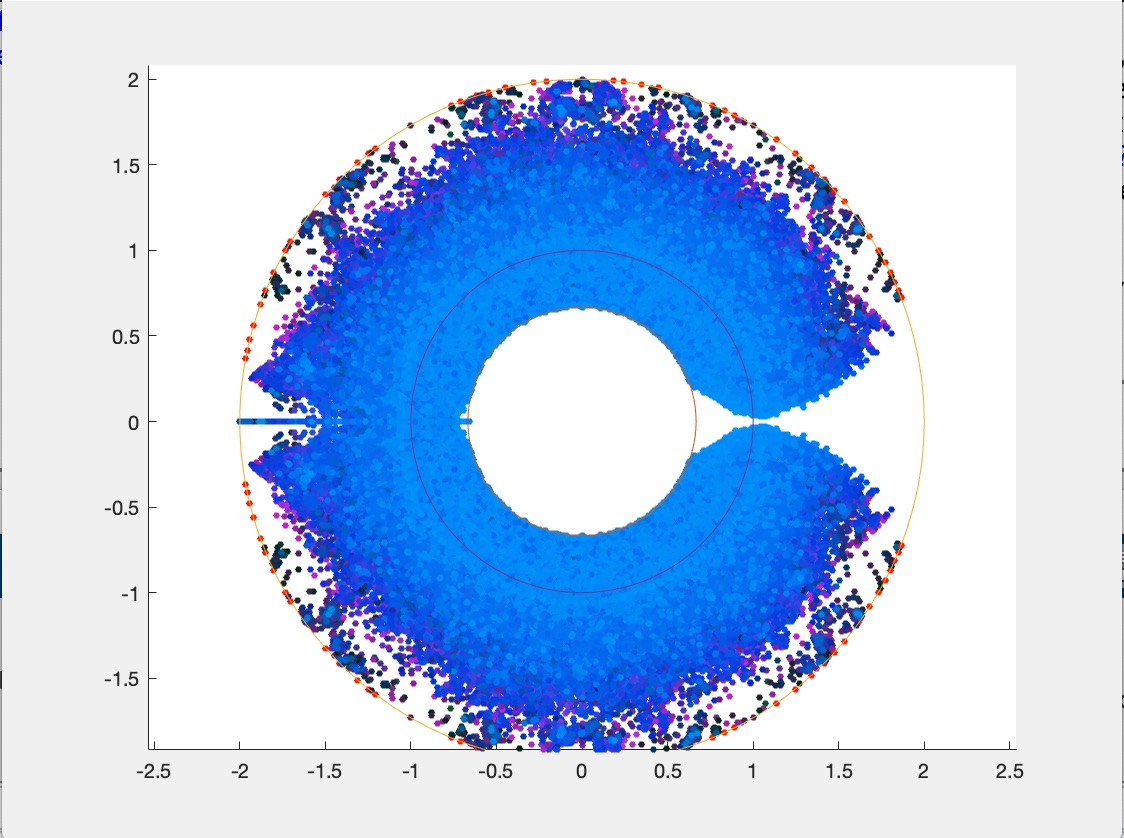}
	\caption{Complex plot of the zeros of $P_N$ for $2\leq N \leq 2^{16}$.}
	\label{fig:zeros}
\end{figure}

\listoffigures

\clearpage 
\setcitestyle{numbers}
\bibliography{Bounds-Collatz-Zeros}

\begin{thebibliography}{5}
\providecommand{\natexlab}[1]{#1}
\providecommand{\url}[1]{\texttt{#1}}
\expandafter\ifx\csname urlstyle\endcsname\relax
  \providecommand{\doi}[1]{doi: #1}\else
  \providecommand{\doi}{doi: \begingroup \urlstyle{rm}\Url}\fi

\bibitem[Anderson et~al.(1981)Anderson, Saff, and Varga]{eksharp}
N.~Anderson, E.~B. Saff, and R.~S. Varga.
\newblock An extension of the {E}neström-{K}akeya theorem and its sharpness.
\newblock \emph{SIAM Journal on Mathematical Analysis}, 12\penalty0
  (1):\penalty0 10–22, 1981.
\newblock \doi{10.1137/0512002}.

\bibitem[Boas()]{boas}
Harold Boas.
\newblock Private communication.

\bibitem[Hohertz and Kalantari(2020)]{hohertz_kalantari}
Matt Hohertz and Bahman Kalantari.
\newblock Collatz polynomials: an introduction with bounds on their zeros.
\newblock 2020.
\newblock URL \url{https://arxiv.org/pdf/2001.00482v2.pdf}.

\bibitem[Hughes and Nikeghbali(2008)]{hughes_nikeghbali_2008}
C.~P. Hughes and A.~Nikeghbali.
\newblock The zeros of random polynomials cluster uniformly near the unit
  circle.
\newblock \emph{Compositio Mathematica}, 144\penalty0 (3):\penalty0 734–746,
  2008.
\newblock \doi{10.1112/s0010437x07003302}.

\bibitem[Lagarias(1985)]{lagarias_1985}
Jeffrey~C. Lagarias.
\newblock The $3x+1$ problem and its generalizations.
\newblock \emph{The American Mathematical Monthly}, 92:\penalty0 3–23, 1985.
\newblock \doi{10.2307/2322189}.
\newblock URL \url{http://www.cecm.sfu.ca/organics/papers/lagarias/index.html}.
\newblock As reprinted in Conference on Organic Mathematics, Canadian Math.
  Society Conference Proceedings, vol. 20, 1997, pp. 305-331, and posted online
  at \url{http://www.cecm.sfu.ca/organics/papers/lagarias/index.html.}

\end{thebibliography}

\end{document}